\documentclass[11pt,reqno]{amsart}

\usepackage{amssymb, amsmath, amsthm}
\usepackage{hyperref}
\usepackage[alphabetic,lite]{amsrefs}
\usepackage{verbatim}
\usepackage{amscd}   % for commutative diagrams
\usepackage[all]{xy} % for complicated commutative diagrams
\usepackage{youngtab} % for Young tableaux
\usepackage{young} % for Young tableaux
\usepackage{ytableau}
\usepackage{tikz}
\usepackage{ mathrsfs }

\newcommand{\defi}[1]{{\upshape\sffamily #1}}
\newcommand{\mf}[1]{\mathfrak{#1}}

\renewcommand{\a}{\alpha}
\renewcommand{\b}{\beta}

\newcommand{\bw}{\bigwedge}
\newcommand{\gl}{\mf{gl}}
\renewcommand{\ll}{\lambda}

\newcommand{\oo}{\otimes}

\newcommand{\Ext}{\operatorname{Ext}}
\newcommand{\GL}{\operatorname{GL}}
\newcommand{\Hom}{\operatorname{Hom}}
\newcommand{\Sym}{\operatorname{Sym}}
\newcommand{\Tor}{\operatorname{Tor}}

\newcommand{\bb}[1]{\mathbb{#1}}
\renewcommand{\rm}[1]{\textrm{#1}}
\newcommand{\mc}[1]{\mathcal{#1}}
\newcommand{\ol}[1]{\overline{#1}}

\newcommand{\scpr}[2]{\left\langle #1,#2 \right\rangle}

\def\lra{\longrightarrow}

\newtheorem{theorem}{Theorem}[section]
\newtheorem*{theorem*}{Theorem}
\newtheorem{lemma}[theorem]{Lemma}

\newtheorem{proposition}[theorem]{Proposition}
\newtheorem{corollary}[theorem]{Corollary}
\newtheorem*{corollary*}{Corollary}

\newtheorem*{syzIaxb*}{Theorem on Syzygies of Rectangular Ideals}

\theoremstyle{definition}

\newtheorem*{definition*}{Definition}
\newtheorem{example}[theorem]{Example}

\newtheorem{problem}[theorem]{Problem}

\theoremstyle{remark}

\newtheorem*{remark*}{Remark}

\numberwithin{equation}{section}

\begin{document}

\title{The syzygies of some thickenings of determinantal varieties}

\author{Claudiu Raicu}
\address{Department of Mathematics, University of Notre Dame, 255 Hurley Hall, Notre Dame, IN 46556\newline
\indent Institute of Mathematics ``Simion Stoilow'' of the Romanian Academy}
\email{craicu@nd.edu}

\author{Jerzy Weyman}
\address{Department of Mathematics, University of Connecticut, Storrs, CT 06269}
\email{jerzy.weyman@uconn.edu}


\date{\today}


\begin{abstract} The vector space of $m\times n$ complex matrices ($m\geq n$) admits a natural action of the group $\GL=\GL_m\times\GL_n$ via row and column operations. For positive integers $a,b$, we consider the ideal $I_{a\times b}$ defined as the smallest $\GL$--equivariant ideal containing the $b$-th powers of the $a\times a$ minors of the generic $m\times n$ matrix. We compute the syzygies of the ideals $I_{a\times b}$ for all $a,b$, together with their $\GL$--equivariant structure, generalizing earlier results of Lascoux for the ideals of minors ($b=1$), and of Akin--Buchsbaum--Weyman for the powers of the ideals of maximal minors ($a=n$). Our methods rely on a nice connection between commutative algebra and the representation theory of the superalgebra $\gl(m|n)$, as well as on our previous calculation of $\Ext$ modules done in the context of describing local cohomology with determinantal support. Our results constitute an important ingredient in the proof by Nagpal--Sam--Snowden of the first non-trivial Noetherianity results for twisted commutative algebras which are not generated in degree one.
\end{abstract}

\maketitle

\section{Introduction}\label{sec:intro}

For positive integers $m\geq n$, we consider the ring $S=\Sym(\bb{C}^m\oo\bb{C}^n)(=\bb{C}[z_{ij}])$ of polynomial functions on the vector space of $m\times n$ matrices with entries in the complex numbers. The ring $S$ admits an action of the group $\GL=\GL_m(\bb{C})\times\GL_n(\bb{C})$, and it decomposes into irreducible $\GL$--representations according to \defi{Cauchy's formula}:
\[S=\bigoplus_{\ll=(\ll_1\geq\ll_2\geq\cdots\geq\ll_n\geq 0)} S_{\ll}\bb{C}^m\oo S_{\ll}\bb{C}^n,\]
where $S_{\ll}$ denotes the \defi{Schur functor} associated to a partition $\ll$. For each $\ll$, we let $I_{\ll}$ denote the ideal in $S$ generated by the irreducible representation $S_{\ll}\bb{C}^m\oo S_{\ll}\bb{C}^n$. Every ideal $I\subset S$ which is preserved by the $\GL$-action is a sum of ideals $I_{\ll}$: such ideals $I$ have been classified and their geometry has been studied by De Concini, Eisenbud and Procesi in the 80s \cite{DCEP}. Nevertheless, their syzygies are still mysterious, and in particular the following problem remains unsolved:

\begin{problem}\label{problem:main}
 Describe the syzygies of the ideals $I_{\ll}$, together with their $\GL$-equivariant structure.
\end{problem}

We expect that a complete solution to this problem will be intimately related to the representation theory of the superalgebra $\gl(m|n)$ \cite{kac}. The connection is explained as follows: the universal enveloping algebra of $\gl(m|n)$ contains as a subalgebra the exterior algebra $\Lambda=\bw\left(\bb{C}^m\oo\bb{C}^n\right)$. Every $\gl(m|n)$ module $P$ can then be thought of as a $\Lambda$-module, which by the BGG correspondence \cite[Chapter~7]{eis-syz} gives rise to a linear complex over the polynomial ring $S$. When $P$ is a Kac module, the corresponding complex is just a Koszul complex, and it is exact. The composition factors of Kac modules however give rise to complexes which are typically no longer exact, and in many cases their homology groups are closely related to the ideals $I_{\ll}$. There is a significant literature related to the character theory of $\gl(m|n)$-modules \cites{serganova,brundan}, and in particular to Kac modules \cites{hughes-king-jeugt,su-hughes-king,su}, and it is our hope that this paper will provide sufficient motivation for a more systematic study of the corresponding complexes from a commutative algebra perspective. The connection with commutative algebra was introduced in work of Akin and Weyman \cites{akin-weyman1,akin-weyman2}. We will use their work together with our previous calculation of $\Ext$ modules \cite{raicu-weyman} in order to answer Problem~\ref{problem:main} for a large class of ideals $I_{\ll}$. This class is large enough to provide an interesting application to the study of Noetherianity properties for twisted commutative algebras \cite{nagpal-sam-snowden}.

The goal of our paper is to give a quick solution to Problem~\ref{problem:main} in the case when $\ll$ is a \defi{rectangular partition}, which means that there exist positive integers $a,b$ such that $\ll_1=\cdots=\ll_a=b$ and $\ll_i=0$ for $i>a$ (alternatively, the \defi{Young diagram} associated to $\ll$ is the $a\times b$ rectangle). In this case we write $\ll=a\times b$ and $I_{\ll}=I_{a\times b}$. One can think of $I_{a\times b}$ as the smallest $\GL$-equivariant ideal which contains the $b$-th powers of the $a\times a$ minors of the generic matrix of indeterminates $Z=(z_{ij})$. What distinguishes the ideals $I_{a\times b}$ among all the $I_{\ll}$'s is that they define a scheme without embedded components, so from a geometric point of view they form the simplest class of $\GL$-equivariant ideals after the reduced (and prime) ideals of minors. Examples of ideals $I_{a\times b}$ include:
\begin{itemize}
 \item $I_{a\times 1}=I_a$, the ideal generated by the $a\times a$ minors of $Z$.
 \item $I_{n\times b}=I_n^b$, the $b$-th power of the ideal $I_n$ of maximal minors of $Z$. 
 \item $I_{1\times b}$, the ideal of $b\times b$ permanents of $Z$: here by a $b\times b$ permanent of $Z$ we mean the permanent of a $b\times b$ matrix obtained by selecting $b$ rows and $b$ columns of $Z$, not necessarily distinct; for instance, when $m=n=2$ we have:
\begin{equation}\label{eq:I1x2}
 I_{1\times 2}=(z_{11}^2,z_{12}^2,z_{21}^2,z_{22}^2,z_{11}z_{12},z_{11}z_{21},z_{12}z_{22},z_{21}z_{22},z_{11}z_{22}+z_{12}z_{21}).
\end{equation}
\end{itemize}

To state our main result, we need to introduce some notation. We write $Rep_{\GL}$ for the representation ring of the group $\GL$, and for a given $\GL$-representation $M$, we let $[M]\in Rep_{\GL}$ denote its class in the representation ring. We let
\begin{equation}\label{eq:defsyzygies}
B_{i,j}(I_{a\times b})=\Tor_i^S(I_{a\times b},\bb{C})_j 
\end{equation}
denote the vector space of $i$-\defi{syzygies} of degree $j$ of $I_{a\times b}$. We encode the syzygies of $I_{a\times b}$ into the \defi{equivariant Betti polynomial}
\begin{equation}\label{eq:defequivBetti}
B_{a\times b}(z,w)=\sum_{i,j\in\bb{Z}} [B_{i,j}(I_{a\times b})]\cdot w^i\cdot z^j\in Rep_{\GL}[z,w],
\end{equation}
so the variable $z$ keeps track of the internal degree, while $w$ keeps track of the homological degree. 

If $r,s$ are positive integers, $\a$ is a partition with at most $r$ parts ($\a_i=0$ for $i>r$) and $\b$ is a partition with parts of size at most $s$ ($\b_1\leq s$), we construct the partition
\begin{equation}\label{eq:defllrsab}
\ll(r,s;\a,\b)=(s+\a_1,\cdots,s+\a_r,\b_1,\b_2,\cdots). 
\end{equation}
This is easiest to visualize in terms of Young diagrams: one starts with an $r\times s$ rectangle, and attach $\a$ to the right and $\b$ to the bottom of the rectangle. If $r=4$, $s=5$, $\a=(4,2,1)$, $\b=(3,2)$, then
\begin{equation}\label{eq:Yngllrsab}
\ll(r,s;\a,\b)=\ytableausetup{smalltableaux,aligntableaux=center}
\ydiagram
{5,5,5,5}
*[*(green)]{5+4,5+2,5+1}
*[*(blue)]{0,0,0,0,3,2}
\end{equation}
We write $\mu'$ for the conjugate partition to $\mu$ (obtained by transposing the Young diagram of $\mu$) and consider the polynomials $h_{r\times s}\in Rep_{\GL}[z,w]$ given by
\begin{equation}\label{eq:defhrs}
h_{r\times s}(z,w)=\sum_{\a,\b}[S_{\ll(r,s;\a,\b)}\bb{C}^m \oo S_{\ll(r,s;\b',\a')}\bb{C}^n]\cdot z^{r\cdot s + |\a|+|\b|}\cdot w^{|\a|+|\b|}, 
\end{equation}
where the sum is taken over partitions $\a,\b$ such that $\a$ is contained in the $\min(r,s)\times(n-r)$ rectangle ($\a_1\leq n-r$, $\a_1'\leq\min(r,s)$) and $\b$ is contained in the $(m-r)\times\min(r,s)$ rectangle ($\b_1\leq\min(r,s)$ and $\b_1'\leq m-r$). We also need to introduce the \defi{Gauss polynomial} ${r+s\choose r}_w\in\bb{Z}[w]$,
\begin{equation}\label{eq:defGausspoly}
{r+s\choose r}_w=\sum_{s\geq t_1\geq\cdots\geq t_r\geq 0} w^{t_1+\cdots+t_r}, 
\end{equation}
which is the generating function for partitions contained inside the $r\times s$ rectangle. Note that ${r+s\choose r}_{w^2}$ is the \defi{Poincar\'e polynomial} of the Grassmannian of $r$-dimensional subspaces of an $(r+s)$-dimensional vector space, and also that ${r+s\choose r}_1={r+s\choose r}$ is the usual binomial coefficient. Our main result is:

\begin{syzIaxb*}[Theorem~\ref{thm:main}]
 The equivariant Betti polynomial of the ideal $I_{a\times b}$ is
\[B_{a\times b}(z,w)=\sum_{q=0}^{n-a} h_{(a+q)\times(b+q)}\cdot w^{q^2+2q}\cdot{q+\min(a,b)-1\choose q}_{w^2}\]
\end{syzIaxb*}

\noindent When $b=1$, this recovers the result of Lascoux on syzygies of determinantal varieties \cite{lascoux}. When $a=n$, we obtain the syzygies of the powers of the ideals of maximal minors, as originally computed by Akin--Buchsbaum--Weyman \cite{akin-buchsbaum-weyman}.

\begin{example}\label{ex:I1x2}
 When $m=n=2$, the ideal $I_{1\times 2}$ from (\ref{eq:I1x2}) has the equivariant Betti polynomial
\[B_{1\times 2}(z,w)=h_{1\times 2}+h_{2\times 3}\cdot w^3,\]
where
\[h_{1\times 2}=[\Sym^2\bb{C}^2\oo\Sym^2\bb{C}^2]\cdot z^2+([\Sym^3\bb{C}^2\oo S_{2,1}\bb{C}^2]+[S_{2,1}\bb{C}^2\oo\Sym^3\bb{C}^2])\cdot z^3\cdot w+[S_{3,1}\bb{C}^2\oo S_{3,1}\bb{C}^2]\cdot z^4\cdot w^2,\]
and
\[h_{2\times 3}=[S_{3,3}\bb{C}^2\oo S_{3,3}\bb{C}^2]\cdot z^6.\]
The equivariant Betti table (where the $(i,j)$-entry is $[B_{i,i+j}(I_{1\times 2})]\in Rep_{\GL}$, represented pictorially in terms of Young diagrams; as in (\ref{eq:Yngllrsab}) we use empty boxes for the $r\times s$ rectangle inside $\ll(r,s;\a,\b)$ and $\ll(r,s;\b',\a')$, green boxes for the partitions $\a,\a'$ and blue boxes for the partition $\b,\b'$) then looks like
\[\ytableausetup{boxsize = 1.25em}
 \begin{array}{|c|c|c|c|}
  \hline
   & & & \\
   \ydiagram{2}\oo\ydiagram{2} & \ydiagram{2}*[*(green)]{2+1}\oo\ydiagram{2}*[*(green)]{0,1} + \ydiagram{2}*[*(blue)]{0,1}\oo\ydiagram{2}*[*(blue)]{2+1} & \ydiagram{2}*[*(green)]{2+1}*[*(blue)]{0,1}\oo\ydiagram{2}*[*(blue)]{2+1}*[*(green)]{0,1} &  - \\
   & & &\\
  \hline
   & & &\\
  - & - & - & \ydiagram{3,3}\oo\ydiagram{3,3}\\
   & & & \\
  \hline
 \end{array}
\]
Taking dimensions of representations ($\dim(\Sym^r\bb{C}^2)=r+1$, $\dim(S_{r,1}\bb{C}^2)=r$, $\dim(S_{r,r}\bb{C}^2)=1$), we get the usual Betti table, which can be verified for instance using Macaulay2 \cite{M2}:
\[
 \begin{array}{|c|c|c|c|}
  \hline
  9 & 16 & 9 & - \\
  \hline
  - & - & - & 1\\
  \hline
 \end{array}
\]
\end{example}

An immediate corollary of Theorem~\ref{thm:main} is a uniform boundedness result for the syzygies of the ideals $I_{a\times b}$: it is easy to see that for fixed $i$, the coefficient of $w^i z^j$ in $B_{a\times b}(z,w)$ is zero when $j$ is large enough, independent on the size $m\times n$ of our matrices. This is translated in \cite{nagpal-sam-snowden} into the fact that over the bivariate twisted commutative algebra $\mc{S}=\Sym(\bb{C}^{\infty}\oo\bb{C}^{\infty})$, the syzygy modules $\Tor_i^{\mc{S}}(I_{a\times b},\bb{C})$ have finite length, which is then used to derive a similar conclusion for the syzygy modules of any finitely generated $\mc{S}$-module. In the slightly different setup of $\Delta$-modules, a similar boundedness result for the syzygies of Segre embeddings \cite[Prop.~5.1]{snowden} has been used by Snowden to prove more refined finiteness properties for the said syzygies.

The proof of Theorem~\ref{thm:main} is based on the following two ingredients:
\begin{itemize}
 \item Joint work of the second author with Akin \cites{akin-weyman1,akin-weyman2}: they introduce and study a family of linear complexes $X^{r\times s}_{\bullet}$, arising via the BGG correspondence from certain simple modules over the superalgebra $\gl(m|n)$. The homology of these linear complexes consists entirely of direct sums of ideals $I_{(r+q)\times (s+q)}$. The polynomials $h_{r\times s}(z,w)$ introduced in (\ref{eq:defhrs}) precisely encode the terms of these linear complexes, or equivalently they encode the characters of the corresponding simple $\gl(m|n)$-modules.
 \item The recent work of the authors on computing local cohomology with support in determinantal ideals: in \cite{raicu-weyman} we compute all the modules $\Ext_S^{\bullet}(I_{a\times b},S)$, together with their $\GL$-equivariant structure.
\end{itemize}

Based on these two ingredients, our strategy is as follows. We obtain a non-minimal resolution of $I_{a\times b}$ via an iterated mapping cone construction involving the linear complexes $X^{(a+q)\times(b+q)}$, $q\geq 0$. We then use the $\GL$-equivariance to conclude that whenever cancellations occur for some of the terms of an $X^{r\times s}$, they must in fact occur for all the terms of $X^{r\times s}$. This implies that the minimal resolution of $I_{a\times b}$ is also built out of copies of $X^{(a+q)\times(b+q)}$, and it remains to determine the number of such copies, as well as their homological shifts. This is done by dualizing the minimal resolution and using the $\GL$-equivariant description of $\Ext^{\bullet}_S(I_{a\times b},S)$. We elaborate on this argument in Section~\ref{sec:mainthm}, after we establish some notational conventions in Section~\ref{sec:prelim}, and collect some preliminary results on functoriality of syzygies, on the complexes $X^{r\times s}$, and on the computation of $\Ext$ modules.

\section{Preliminaries}\label{sec:prelim}

\subsection{Representation Theory {\cite{ful-har}, \cite[Ch.~2]{weyman}}}\label{subsec:repthy}
If $W$ is a complex vector space of dimension $\dim(W)=n$, a choice of basis determines an isomorphism between $\GL(W)$ and the group $\GL_n(\bb{C})$ of $n\times n$ invertible matrices. We will refer to $n$--tuples $\ll=(\ll_1,\cdots,\ll_n)\in\bb{Z}^n$ as \defi{weights} of the corresponding maximal torus of diagonal matrices. We say that $\ll$ is a \defi{dominant weight} if $\ll_1\geq\ll_2\geq\cdots\geq\ll_n$. Irreducible representations of $\GL(W)$ are in one-to-one correspondence with dominant weights $\ll$. We denote by $S_{\ll}W$ the irreducible representation associated to $\ll$. We write $|\ll|$ for the total size $\ll_1+\cdots+\ll_n$ of $\ll$.
 
When $\ll$ is a dominant weight with $\ll_n\geq 0$, we say that $\ll$ is a \defi{partition} of $r=|\ll|$. We will often represent a partition via its associated \defi{Young diagram} which consists of left--justified rows of boxes, with $\ll_i$ boxes in the $i$--th row: for example, the Young diagram associated to $\ll=(5,2,1)$ is
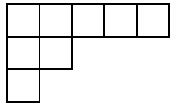
\[\Yvcentermath1\yng(5,2,1)\]
Note that when we're dealing with partitions we often omit the trailing zeros. We define the \defi{length} of a partition $\ll$ to be the number of its non-zero parts, and denote it by $l(\ll)$. If $l(\ll)>\dim(W)$ then $S_{\ll}W=0$. The \defi{transpose} $\ll'$ of a partition $\ll$ is obtained by transposing the corresponding Young diagram. For the example above, $\ll'=(3,2,1,1,1)$, $l(\ll)=3$ and $l(\ll')=5$. If $\mu$ is another partition, we write $\mu\subset\ll$ to indicate that $\mu_i\leq\ll_i$ for all $i$, and say that $\mu$ is \defi{contained} in $\ll$.

For a pair of finite dimensional vector spaces $F,G$, we write $\GL(F,G)$ (or simply $\GL$ when $F,G$ are understood) for the group $\GL(F)\times\GL(G)$. If $M$ is a $\GL(F,G)$-representation, we write
\[\scpr{S_{\ll}F\oo S_{\mu}G}{M}\]
for the multiplicity of the irreducible $\GL$-representation $S_{\ll}F\oo S_{\mu}G$ inside $M$. If $M^{\bullet}$ is a cohomologically graded module, then we record the occurrences of $S_{\ll}F\oo S_{\mu}G$ inside the graded components of $M^{\bullet}$ by
\begin{equation}\label{eq:multsMbullet}
\scpr{S_{\ll}F\oo S_{\mu}G}{M^{\bullet}}=\sum_{i\in\bb{Z}}\scpr{S_{\ll}F\oo S_{\mu}G}{M^i}\cdot w^i,
\end{equation}
where the variable $w$ encodes the cohomological degree (note a slight difference from (\ref{eq:defequivBetti}), where $w$ was used for homological degree).

\subsection{Functoriality of syzygies}\label{subsec:functoriality}

It will be useful to think of the polynomial ring $S=\Sym(\bb{C}^m\oo\bb{C}^n)$ as a functor $S$ which assigns to a pair $(F,G)$ of finite dimensional vector spaces the polynomial ring $S(F,G)=\Sym(F\oo G)$. For each $a,b$ we obtain functors $I_{a\times b}$ which assign to $(F,G)$ the corresponding ideal $I_{a\times b}(F,G)\subset S(F,G)$. The syzygy modules in (\ref{eq:defsyzygies}) become functors $B_{i,j}^{a\times b}(-,-)$, defined by
\[B_{i,j}^{a\times b}(F,G)=\Tor_i^{S(F,G)}(I_{a\times b}(F,G),\bb{C})_j.\]
In fact, each $B_{i,j}^{a\times b}$ is a \defi{polynomial functor} in the sense of \cite[Ch.~I, Appendix A]{macdonald}. As such they decompose into a (usually infinite) direct sum indexed by pairs of partitions
\begin{equation}\label{eq:functBij}
B_{i,j}^{a\times b}(-,-)=\bigoplus_{|\ll|=|\mu|=j}(S_{\ll}(-)\oo S_{\mu}(-))^{\bigoplus m_{\ll,\mu}}. 
\end{equation}
When evaluating $B_{i,j}^{a\times b}$ on a pair of vector spaces $(F,G)$, the only terms on the right hand side of (\ref{eq:functBij}) that survive are the ones for which $l(\ll)\leq\dim(F)$ and $l(\mu)\leq\dim(G)$. The multiplicities $m_{\ll,\mu}$ for such pairs $(\ll,\mu)$ are then determined by the $\GL(F,G)$-equivariant structure of $B_{i,j}^{a\times b}(F,G)$. In particular, knowing the $\GL$-equivariant structure for the syzygies of $I_{a\times b}(\bb{C}^m,\bb{C}^n)$ determines the syzygies of $I_{a\times b}(F,G)$ for all pairs of vector spaces $(F,G)$ with $\dim(F)\leq m$, $\dim(G)\leq n$.

\subsection{The linear complexes $X^{r\times s}$ of Akin and Weyman}\label{subsec:akinweyman}

In \cites{akin-weyman1,akin-weyman2}, Akin and the second author construct linear complexes $X^{r\times s}=X^{r\times s}_{\bullet}(F,G)$ which depend functorially on a pair of finite dimensional vector spaces $(F,G)$. The terms in the complex are given (using notation (\ref{eq:defllrsab})) by
\begin{equation}\label{eq:defXrs}
X^{r\times s}_i(F,G)=\left(\bigoplus_{\substack{|\a|+|\b|=i \\ \a_1',\b_1\leq\min(r,s)}}S_{\ll(r,s;\a,\b)}F\oo S_{\ll(r,s;\b',\a')}G\right)\oo S(F,G).
\end{equation}
Note that since $S_{\ll}W=0$ when $l(\ll)>\dim(W)$, only finitely many of the terms $X^{r\times s}_i(F,G)$ in (\ref{eq:defXrs}) are non-zero for a given pair $(F,G)$. More precisely, we must have $\a_1\leq\dim(G)-r$, $\b_1'\leq\dim(F)-r$, so $|\a|\leq\min(r,s)\cdot(\dim(G)-r)$, $|\b|\leq\min(r,s)\cdot(\dim(F)-r)$, $i\leq\min(r,s)\cdot(\dim(F)+\dim(G)-2r)$. We can rewrite (\ref{eq:defhrs}) as
\[h_{r\times s}(z,w)=\sum_{i = 0}^{\min(r,s)\cdot(m+n-2\cdot r)}[X^{r\times s}_i(\bb{C}^m,\bb{C}^n)_{r\cdot s+i}]\cdot z^{r\cdot s+i}\cdot w^i,\]
where $X^{r\times s}_i(\bb{C}^m,\bb{C}^n)_{r\cdot s+i}$ is the vector space of minimal generators of the free module $X^{r\times s}_i(\bb{C}^m,\bb{C}^n)$. The complex $X^{(a+q)\times(1+q)}_{\bullet}(\bb{C}^m,\bb{C}^n)$ can be identified with the $q$-th linear strand of the Lascoux resolution of the ideal of $a\times a$ minors of the generic $m\times n$ matrix. In this paper we'll see that more generally, the complexes $X^{(a+q)\times (b+q)}$, $q\geq 0$, form the building blocks of the minimal resolutions of the ideals $I_{a\times b}$.

The complex $X^{r\times s}=X^{r\times s}_{\bullet}(\bb{C}^m,\bb{C}^n)$ corresponds via BGG to the irreducible $\mf{gl}(m|n)$-module of lowest weight $(s^r,0^{m-r}|s^r,0^{n-r})$ (thought of as a module over the exterior algebra). In \cite{akin-weyman2} the homology of the complexes $X^{r\times s}$ is shown to consist of direct sums of the rectangular ideals $I_{(r+q)\times (s+q)}$. To state this more precisely, we need to introduce some notation. We denote by $P(r,s;i)$ the number of partitions of $i$ contained in the $r\times s$ rectangle. The Gauss polynomial defined in (\ref{eq:defGausspoly}) is then
\[{r+s\choose r}_w=\sum_{i=0}^{r\cdot s} P(r,s;i)w^i.\]

\begin{theorem}[{\cite[Thm.~2]{akin-weyman2}}]\label{thm:homologyX}
With the above notation, the homology groups of $X^{r\times s}_{\bullet}$ are
\begin{enumerate}
\item $H_{2j+1}(X^{r\times s}_{\bullet})=0$;
\item $H_{2j}(X^{r\times s}_{\bullet})=\displaystyle\bigoplus_{q = 0}^j I_{(r+q)\times(s+q)}^{\oplus P(q,\min(r,s)-1;j-q)}$.
\end{enumerate}
\end{theorem}

In \cite{akin-weyman2} the projective dimension of the ideals $I_{a\times b}$ is calculated. The calculation of $\Ext$ modules in \cite[Thm.~4.3]{raicu-weyman} in fact allows one to compute the projective dimension and regularity for all the ideals $I_{\ll}$, i.e the shape of their minimal resolution. More work is however necessary in order to completely determine the syzygies.

\subsection{The $\Ext$ modules $\Ext_S^{\bullet}(I_{a\times b},S)$}

In \cite[Theorem~4.3]{raicu-weyman} we determined the decomposition into irreducible $\GL$-representations for all the modules $\Ext_S^{\bullet}(I_{\ll},S)$. In the case when $\ll$ is a rectangular partition, we obtain the following consequence which will be useful for our calculation of syzygies.

\begin{theorem}\label{thm:ExtIab}
 Assume that $m=n$ and write $q=n-a$, $S=S(\bb{C}^n,\bb{C}^n)$, $I_{a\times b}=I_{a\times b}(\bb{C}^n,\bb{C}^n)$, $\GL=\GL(\bb{C}^n,\bb{C}^n)$. The occurrences of the irreducible $\GL$--representation $S_{(-b-q)^n}\bb{C}^n\oo S_{(-b-q)^n}\bb{C}^n$ inside\linebreak $\Ext_S^{\bullet}(I_{a\times b},S)$ (see (\ref{eq:multsMbullet})) are encoded as
\[\scpr{S_{(-b-q)^n}\bb{C}^n\oo S_{(-b-q)^n}\bb{C}^n}{\Ext_S^{\bullet}(I_{a\times b},S)}=w^{q^2+2q}\cdot{q+\min(a,b)-1 \choose q}_{w^2}.\]
\end{theorem}

\section{The syzygies of the ideals $I_{a\times b}$}\label{sec:mainthm}

We now proceed to state and prove the main result of our paper:

\begin{theorem}\label{thm:main}
 The equivariant Betti polynomial of $I_{a\times b}\subset\Sym(\bb{C}^m\oo\bb{C}^n)$, $m\geq n$, is
\[B_{a\times b}(z,w)=\sum_{q=0}^{n-a} h_{(a+q)\times(b+q)}\cdot w^{q^2+2q}\cdot{q+\min(a,b)-1\choose q}_{w^2},\]
where $h_{r\times s}=h_{r\times s}(z,w)$ is as defined in (\ref{eq:defhrs}).
\end{theorem}

We prove Theorem \ref{thm:main} in a few stages. We first note that by functoriality (Section~\ref{subsec:functoriality}) it is enough to prove the theorem in the case $m=n$, which we assume for the remainder of this section. For brevity, we will say that a complex $Y$ is \defi{filtered by the linear complexes} $X^i$ if $Y$ admits a filtration by subcomplexes in such a way that the subquotients are isomorphic to $X^i$. Equivalently, $Y$ can be built out of $X^i$ via an iterated mapping cone construction. We have

\begin{proposition}\label{prop:nonminresI}
The ideal $I_{a\times b}$ has a (not necessarily minimal) free $\GL$-equivariant resolution over $S$, denoted $Y^{a\times b}$, which is filtered by the complexes $X^{(a+q)\times (b+q)}$.
\end{proposition}

\begin{proof} We prove by descending induction on $q$ that $I_{(a+q)\times (b+q)}$ admits a (not necessarily minimal) resolution $Y^{(a+q)\times(b+q)}$ which is filtered by complexes $X^{(a+q')\times(b+q')}$ with $q'\geq q$. If $q=n-a$ then $I_{(a+q)\times(b+q)}=I_{n\times (b+n-a)}$ coincides with $X^{n\times(b+n-a)}$: they are both isomorphic to a free module of rank one, generated by the $(b+n-a)$-th power of the determinant of the generic $n\times n$ matrix. Assuming now that the result is true for the ideals $I_{(a+q)\times (b+q)}$ with $q>q_0$, we'll prove it for $q=q_0$ to finish the inductive argument. By Theorem~\ref{thm:homologyX} the higher homology of the linear complex $X^{(a+q_0)\times(b+q_0)}$ consists of direct sums of ideals $I_{(a+q)\times(b+q)}$, $q>q_0$, and $H_0(X^{(a+q_0)\times(b+q_0)})=I_{(a+q_0)\times(b+q_0)}$. We can therefore construct a resolution $Y^{(a+q_0)\times(b+q_0)}$ of $I_{(a+q_0)\times(b+q_0)}$ as a mapping cone of the maps from the complexes $Y^{(a+q)\times (b+q)}$, $q>q_0$, to the complex $X^{(a+q_0)\times(b+q_0)}$ that cancel its higher homology.
\end{proof}

Let $Y^{a\times b}$ be a non-minimal $\GL$-equivariant resolution of the ideal $I_{a\times b}$ as in Proposition~\ref{prop:nonminresI}. We can minimize $Y^{a\times b}$ by making appropriate cancellations.
Notice that since the generators of the free modules appearing in $X^{(a+q)\times(b+q)}$ and $X^{(a+q')\times(b+q')}$ don't share isomorphic irreducible $\GL$-subrepresentations for $q\neq q'$, the only cancellations that can occur are between the terms in various copies of the same $X^{(a+q)\times(b+q)}$.

\begin{lemma}\label{lem:EndXrs}
 Any $\GL(F,G)$-equivariant endomorphism of $X^{r\times s}_{\bullet}(F,G)$ is a multiple of the identity.
\end{lemma}

\begin{proof} Let $\psi$ denote a $\GL$-equivariant endomorphism of $X^{r\times s}$, and write $\psi_i$ for its component in homological degree $i$. By $\GL$-equivariance and using the decomposition (\ref{eq:defXrs}), we have $\psi_i=\bigoplus_{\a,\b}\psi_{\a,\b}$, where $\psi_{\a,\b}$ is the restriction of $\psi_i$ to the free submodule $X^{r\times s}_{\a,\b}$ generated by the irreducible representation $S_{\ll(r,s;\a,\b)}F\oo S_{\ll(r,s;\b',\a')}G$. Such an endomorphism is necessarily a multiple of the identity. Writing $\psi_{\a,\b}=\cdot c_{\a,\b}$, it suffices to show that all $c_{\a,\b}$ are the same. We prove this by induction on $i=|\a|+|\b|$.

Consider $(\a,\b)$ with $i=|\a|+|\b|>0$, and consider a pair $(\ol{\a},\ol{\b})$ with $|\ol{\a}|+|\ol{\b}|=i-1$, such that the restriction of the differential $\partial_i:X^{r\times s}_i\to X^{r\times s}_{i-1}$ to
\[X^{r\times s}_{\a,\b}\overset{\partial_i}{\lra}X^{r\times s}_{\ol{\a},\ol{\b}}\]
is non-zero: such a pair exists since otherwise $S_{\ll(r,s;\a,\b)}F\oo S_{\ll(r,s;\b',\a')}G$ would contribute to the homology of $X^{r\times s}_{\bullet}$ (note that this representation is not a coboundary, since the complex $X^{r\times s}_{\bullet}$ is minimal), which would contradict Theorem~\ref{thm:homologyX}. Since $\psi$ commutes with the differentials, we have a commutative diagram
\[
 \xymatrix{
  X^{r\times s}_{\a,\b} \ar[r]^{\partial_i} \ar[d]_{\cdot c_{\a,\b}} & X^{r\times s}_{\ol{\a},\ol{\b}} \ar[d]^{\cdot c_{\ol{\a},\ol{\b}}} \\
  X^{r\times s}_{\a,\b} \ar[r]^{\partial_i} & X^{r\times s}_{\ol{\a},\ol{\b}} \\
}
\]
Since $\partial_i\neq 0$, it follows that $c_{\a,\b}=c_{\ol{\a},\ol{\b}}$, and we conclude by induction.
\end{proof}

The preceding discussion implies the following

\begin{corollary}\label{cor:minresIab}
The minimal resolution of $I_{a\times b}$ is filtered by the complexes $X^{(a+q)\times (b+q)}$, $q\geq 0$. In particular, there exist polynomials $M_{a\times b}^q(w)$ which account for the multiplicities of the complexes $X^{(a+q)\times(b+q)}$ in the minimal resolution of $I_{a\times b}$, as well as for their homological shifts, i.e.
\[B_{a\times b}(z,w)=\sum_{q=0}^{n-a} h_{(a+q)\times(b+q)}\cdot M_{a\times b}^q(w).\]
\end{corollary}

\begin{proof}
 By Proposition~\ref{prop:nonminresI}, $I_{a\times b}$ admits a non-minimal resolution $Y^{a\times b}$, filtered by the complexes $X^{(a+q)\times(b+q)}$. To get to the minimal resolution of $I_{a\times b}$, one must perform appropriate cancellations, which can only occur between copies of the same $X^{(a+q)\times(b+q)}$. By Lemma~\ref{lem:EndXrs}, such cancellations occur either for all the terms in $X^{(a+q)\times(b+q)}$, or for none. This shows that the minimal resolution of $I_{a\times b}$ is filtered by (possibly fewer) copies of $X^{(a+q)\times(b+q)}$, from which the remaining part of the conclusion follows formally.
\end{proof}

We are now ready to prove the main result of the paper:

\begin{proof}[Proof of Theorem~\ref{thm:main}]
 It remains to calculate the polynomials $M_{a\times b}^q(w)$. We fix $q$ and shrink $n$ if necessary to assume that $n=a+q$ (see Section~\ref{subsec:functoriality}), so $X^{(a+q)\times(b+q)}=X^{n\times(b+q)}$ consists of a single free module, generated by the irreducible $\GL$-representation $S_{(b+q)^n}\bb{C}^n\oo S_{(b+q)^n}\bb{C}^n$. Dualizing the minimal resolution $Y$ of $I_{a\times b}$ and computing the cohomology $\Ext_S^{\bullet}(I_{a\times b},S)$ of the resulting complex $Y^{\vee}$, we get
\begin{enumerate}
 \item[(a)] each occurrence of $X^{n\times(b+q)}$ in $Y$ yields a copy of $S_{(-b-q)^n}\bb{C}^n\oo S_{(-b-q)^n}\bb{C}^n$ in $\Ext_S^{\bullet}(I_{a\times b},S)$;
 \item[(b)] the only occurrences of $S_{(-b-q)^n}\bb{C}^n\oo S_{(-b-q)^n}\bb{C}^n$ inside $\Ext_S^{\bullet}(I_{a\times b},S)$ arise in this way.
\end{enumerate}
To prove (a), note that there are no non-zero maps going into the free module $X^{n\times(b+q)}$, so its dual $\Hom_S(X^{n\times(b+q)},S)$ will consists entirely of cocycles in $Y^{\vee}$. Since $Y^{\vee}$ is minimal, the space $S_{(-b-q)^n}\bb{C}^n\oo S_{(-b-q)^n}\bb{C}^n$ of minimal generators of $(X^{n\times(b+q)})^{\vee}=\Hom_S(X^{n\times(b+q)},S)$ contains no coboundaries, so (a) follows. If (b) failed, one could find a free submodule $M^*\oo S$ in $Y^{\vee}$, containing $S_{(-b-q)^n}\bb{C}^n\oo S_{(-b-q)^n}\bb{C}^n$, where $M$ is an irreducible $\GL$-representation appearing as a subspace of minimal generators in some complex $X^{(a+q')\times(b+q')}$, $q'<q$. The condition $S_{(-b-q)^n}\bb{C}^n\oo S_{(-b-q)^n}\bb{C}^n\subset M^*\oo S$ implies that $M$ appears as a subrepresentation of $S_{(b+q)^n}\bb{C}^n\oo S_{(b+q)^n}\bb{C}^n \oo S$. This can only happen if $M=S_{\ll}\bb{C}^n\oo S_{\mu}\bb{C}^n$, where $\ll,\mu$ are partitions containing the $n\times(b+q)$ rectangle. By (\ref{eq:defllrsab}), $M$ can only occur inside $X^{n\times(b+q)}$.

It follows from (a) and (b) that there is a one-to-one correspondence between occurrences of $X^{n\times(b+q)}$ inside $Y$ and those of $S_{(-b-q)^n}\bb{C}^n\oo S_{(-b-q)^n}\bb{C}^n$ inside $\Ext_S^{\bullet}(I_{a\times b},S)$, and moreover this correspondence replaces homological shifts with cohomological shifts. We get (see (\ref{eq:multsMbullet}) and the remark following it) that
\[M_{a\times b}^q(w)=\scpr{S_{(-b-q)^n}\bb{C}^n\oo S_{(-b-q)^n}\bb{C}^n}{\Ext_S^{\bullet}(I_{a\times b},S)}\overset{Thm.~\ref{thm:ExtIab}}{=}w^{q^2+2q}\cdot{q+\min(a,b)-1 \choose q}_{w^2}.\]
This concludes the proof of Theorem~\ref{thm:main}.
\end{proof}

\section*{Acknowledgments} 
This work was initiated while we were visiting the Mathematical Sciences Research Institute, for whose hospitality we are grateful. Experiments with the computer algebra software Macaulay2 \cite{M2} have provided numerous valuable insights. Raicu acknowledges the support of NSF grant DMS-1458715. Weyman acknowledges the support of the Alexander von Humboldt Foundation and of NSF grant DMS-1400740.

	%%%%%%%%%%%%%%%%%%%%%%%%%%%%%%%%%%%%%%%%%%%%%%%%%%%%%%%%%%%%%%%%%%%%%%%%
	%%%%%%%%%%%%%%%   		Bibliography				%%%%%%%%%%%%%%%%%%%%
	%%%%%%%%%%%%%%%%%%%%%%%%%%%%%%%%%%%%%%%%%%%%%%%%%%%%%%%%%%%%%%%%%%%%%%%%

\end{document}